\documentclass[reqno]{amsart}
\usepackage{amsfonts, graphicx, amsmath, amssymb, color, verbatim}
\usepackage[all, 2cell]{xypic}
\usepackage[margin = 1in]{geometry}
\usepackage{hyperref}
\newtheorem{theorem}{Theorem}
\newtheorem{lemma}[theorem]{Lemma}
\newtheorem{proposition}[theorem]{Proposition}

\newtheorem{corollary}[theorem]{Corollary} 
\newtheorem{remark}[theorem]{Remark}

\numberwithin{equation}{section}
\numberwithin{theorem}{section}

\newcommand \Id {\mathrm{Id}}

\newcommand \p {\partial}

\newcommand \lra {\longrightarrow}

\newcommand \R {\mathbb{R}}
\newcommand \N {\mathbb{N}}
\newcommand \C {\mathbb{C}}
\newcommand \Sph {\mathbb{S}}
\newcommand \Vol {\mathrm{Vol}}

\renewcommand{\hat}[1]{\widehat{#1}}

\makeatletter

\newcommand{\Rmnum}[1]{\expandafter\@slowromancap\romannumeral #1@}
\makeatother

\newcommand{\la}{\langle}
\newcommand{\ra}{\rangle}
\newcommand{\lp}{\left(}
\newcommand{\rp}{\right)}
\newcommand{\set}[1]{\left\{ #1 \right\} }

\newcommand\paperintro%
        {%
         }
\newcommand\paperbody%
        {%
         }

\DeclareMathOperator \Tr {Tr}

\DeclareMathOperator \supp {supp}

\title{Equidistribution of phase shifts in obstacle scattering}
\author{Jesse Gell-Redman}
\address{School of Mathematics and Statistics, University of Melbourne}
\email{jgell@unimelb.edu.au}
\author{Maxime Ingremeau}
\address{Institut de Recherche Math\'ematique Avanc\'ee, Universit\'e de Strasbourg}
\email{maxime.ingremeau@math.unistra.fr}

\begin{document}

\begin{abstract}
For scattering off a smooth, strictly convex obstacle
$\Omega \subset \mathbb{R}^d$ with positive curvature, we show that the
eigenvalues of the scattering matrix -- the phase shifts -- equidistribute on the unit
circle as the frequency $k \to \infty$ at a rate proportional to
$k^{d - 1}$, under a standard condition on the set of closed
orbits of the billiard map in the interior.  Indeed, in any sector $S
\subset \mathbb{S}^1$ not containing $1$, there are $c_d |S| \Vol(\p
\Omega)\ k^{d - 1} + o(k^{d-1})$ eigenvalues for $k$ large, where $c_d$ is a
constant depending only on the dimension. Using this result, the two term
asymptotic expansion for the counting function of Dirichlet
eigenvalues, and a spectral-duality result of Eckmann-Pillet, we then
give an alternative proof of the two term asymptotic of the
total scattering phase due to Majda-Ralston \cite{MajdaRalston1978}.
\end{abstract}

\maketitle

\section{Introduction}

Let $\Omega \subset \mathbb{R}^{d}$ denote a strictly convex domain
whose boundary $\p \Omega$ is smooth and has positive sectional curvature. We shall write $\Omega^c := \R^d\backslash \Omega$.
It is well-known (see for instance \cite[\S 5]{GST} or \cite[\S4.4]{dyatlovmathematical}) that for any $k>0$ and any $\phi_{in}\in C^\infty (\Sph^{d-1})$, there is a
unique solution $u\in C^\infty(\overline{\Omega^c})$ to the Dirichlet problem
$$
(\Delta + k^2)u = 0 \qquad u \rvert_{\p \Omega} = 0
$$
such that
\begin{equation}\label{eq:defscattering}
u(x)=|x|^{-(d-1)/2}\big{(} e^{-i k|x|} \phi_{in}(-\hat{x}) +
e^{i k |x|} \phi_{out}(\hat{x}) \big{)} + O_{|x|\rightarrow \infty}(|x|^{-(d+1)/2}),
\end{equation}
where we write $\hat{x}= \frac{x}{|x|}\in \Sph^{d-1}$ and $\Delta =
\sum_{i = 1}^d \p_{x_i}^2$.  In particular $\phi_{out}$ is determined
by $\phi_{in}$ and we define the \emph{scattering matrix} $S(k)$, which depends on $k$ and
$\Omega$, by $$S(k)(\phi_{in}) := e^{i\pi (d-1)/2} \phi_{out}.$$
$S(k)$ extends to a unitary operator acting on
$L^2(\Sph^{d-1})$ with the property that $S(k) - \Id$ is trace class
\cite{taylor:vol2, RSIII}. Therefore, for any $k>0$, $S(k)$ has purely
discrete spectrum, accumulating only at 1, which we denote by
$\sigma(S(k)) := \{e^{i
  \beta_{k,n}} \}$. Our aim in this paper is to study
the asymptotic distribution of the $e^{i \beta_{k,n}}$ as
$k\rightarrow \infty$.

Our main result is an estimate for the number of phase shifts in a
sector $S \subset \mathbb{S}^1 \setminus \{ 1 \}$ as $k \to \infty$.
Define the counting function
$$
N_k(\phi_0, \phi_1, \Omega) = N_k(\phi_0, \phi_1) := \# \{ e^{i
  \beta_{k,n}} \in \sigma(S(k))  : \phi_0 <
\beta_{k, n} < \phi_1,\ \mathrm{mod}\ 2\pi \}.
$$
where the eigenvalues are counted according to multiplicity.
Letting $\omega_{d-1} = |B^{d-1}|$ where $B^{d-1}$ is the
unit ball in $\mathbb{R}^{d-1}$, we will prove
\begin{equation}
    \label{eq:number}
 N_{k}(\phi_{0}, \phi_{1}) = \frac{\omega_{d-1}}{(2\pi)^{d-1}} \Big{(}\frac{\phi_{1} - \phi_{0}}{2\pi} \Big{)}
 \Vol(\p \Omega) k^{d-1} + o(k^{d-1}).
  \end{equation}
In particular, the phase shifts accumulate in each sector $S$ at a rate
proportional to $k^{d-1}$ as $k \to \infty$ times $\Vol(\p \Omega)
|S|$.  The estimate
in \eqref{eq:number} follows from Theorem \ref{th: main theorem}, see
Section \ref{sec:proofs}.

To study
the asymptotic distribution of the phase shifts, consider the measure $\mu_{k}$ on the circle $\mathbb{S}^{1}$,
defined for continuous functions $f \colon \mathbb{S}^{1} \lra \mathbb{C}$ by
\begin{equation}
  \label{eq:measure}
  \la \mu_{k}, f \ra = \Big{(} \frac{2\pi}{k}\Big{)}^{d - 1} \sum_{\sigma(S(k))}
  f(e^{i \beta_{k,n}}).
\end{equation}
Note that $ \la \mu_{k}, f \ra$ is finite if $1 \not \in \supp f$. The following theorem describes the behavior $\mu_k$ as $k\rightarrow \infty$, provided (\ref{eq: hypVol2}) holds, which is a standard assumption on the volume of the periodic points of the inside billiard map. Note that this assumption holds if our smooth convex obstacle, is generic (see the discussion at the end of Section \ref{sec:dynamics}).
\begin{theorem}\label{th: main theorem}
Let $\Omega\subset \R^d$ be a smooth strictly convex open set with positive sectional curvature, such that (\ref{eq: hypVol2}) holds. Then  for any $f
  \colon \mathbb{S}^{1} \lra \mathbb{C}$ with $\supp f \cap \set{1} =
  \varnothing$, we have
  \begin{equation}
    \label{eq:limit_measure}
\lim \limits_{k\rightarrow \infty}  \la \mu_{k}, f \ra =
\frac{\Vol(\p \Omega) \omega_{d-1}}{2\pi}    \int_{0}^{2\pi}f(e^{i\theta}) d\theta
  \end{equation}
\end{theorem}

\begin{remark}
  The factor in front of the integral in \eqref{eq:limit_measure} arises as the volume of the
  `interacting region' in phase space of incoming rays from the sphere
  at infinity that make contact with the obstacle.  See Section
  \ref{sec:dynamics} for further description of the classical dynamics.  In
  \cite{MR3335243}, in which the first author and collaborators
  studied the same problem for semiclassical potential scattering,
  they defined a measure $\mu_h$, depending on a semiclassical
  parameter $h \to 0$, analogously to the measure in
  \eqref{eq:measure} except they included the volume of the
  interacting region.  Here we prefer not to, so that the dependence
  on the interacting region appears explicitly in the limit measure.
\end{remark}

As an application of the equidistribution of the measure $\mu_k$, we
will give an alternative proof of the following result of
Majda-Ralston, generalized by Melrose and then by Robert,
regarding the asymptotic development of the total scattering phase
\begin{equation}
  \label{eq:scattering-phase}
  s(k) = i \log \det S(k).
\end{equation}
The scattering phase $s(k)$  can be defined in a natural way so
that $s(k) \in C^{\infty}((0, \infty))$.
\begin{theorem}[\cite{MajdaRalston1978, MelroseWeylExterior, Robert1996}]\label{thm:scattering-phase}
Let $\Omega$ be a smoothly bounded, strictly convex obstacle whose set of periodic
billiard trajectories has measure zero.  Then
\begin{equation}
  \label{eq:scattering-phase-asymptotics}
  s(k) = \frac{\omega_d}{(2\pi)^{d-1}} \Vol(\Omega) k^d +
  \frac{\omega_{d - 1}}{4 (2\pi)^{d-2}}  \Vol(\p \Omega) k^{d -1} + o(k^{d-1}).
\end{equation}
 \end{theorem}
In fact, Theorem \cite{MelroseWeylExterior,Robert1996} holds for
\textit{all smoothly bounded, compact domains} satisfying the stated
assumption on the periodic trajectories.

As we describe in Section \ref{sec:proofs}, the novelty in our proof
comes from its use of the explicit relationship between 
the counting function for the Dirichlet eigenvalues,
\begin{equation}
  \label{eq:counting-function}
N_D(\lambda_0) := \# \{ 0 < \lambda < \lambda_0 :  \exists \phi \in L^2(\Omega),\ \phi\rvert_{\p \Omega} = 0,\ \Delta \phi
= - \lambda^2 \phi, \phi \neq 0 \}. 
\end{equation}
and the scattering phase which arises from the spectral duality result of
Eckmann-Pillet \cite{EP1995}.  Indeed, note that the leading order
term in \eqref{eq:scattering-phase-asymptotics} is $2\pi$ times the
leading order term in Weyl's law \cite{ivrii1980second}, which is to
be expected since, as explained in Section \ref{sec:proofs},
`inside-outside' duality says that a phase shift makes a complete
rotation of the unit circle for each Dirichlet eigenvalue of
$\Omega$. 

\medskip

\subsubsection*{Relation to other works}
Since the pioneering works of
Birman, Sobolev, and Yafaev (see for example
\cite{SY1985, BY1984}), there has been a wealth of literature on the asymptotic behavior of the
scattering matrix at high energy, in particular about the distribution
of phase shifts.  In semi-classical potential scattering, an analogous result for compactly supported
potentials was proven by the first author, Hassell, and Zelditch in
\cite{MR3335243} for non-trapping potentials, and was generalized to
trapping potentials by the second author in \cite{I2016}.  See \cite{MR3335243} for a complete literature review of phase
shift asymptotics for potential scattering. The behaviour of the phase shifts in the semi-classical limit has been studied in various settings: for magnetic potentials (\cite{BP2012}), for scattering by radially symmetric potentials, in
\cite{DGHH2013}, near resonant energies in \cite{nakamura2014spectrum}...  

The idea of using trace formulae to analyze the
asymptotics of the spectra comes from \cite{Z1992,Z1997}, and was the starting point of \cite{MR3335243}, \cite{I2016} and of the present paper. The main tool we use here is the Kirchhoff approximation, which was proven in its optimal form in \cite{Melrose-Taylor-near-peak}. Finally, our proof is simplified by describing the micro-local properties of the scattering matrix in terms of its action on Gaussian states, an approach which was introduced in \cite{I2016} for potential scattering.

There do exist perturbations of the free Hamiltonian for which \emph{the
phase shifts do not equidistribute.}  Indeed, for Schr\"odinger
operators of the form $\Delta + V$ where $V \sim v_0(\hat x) / | x
|^\alpha$, $\alpha > d$, the first author and Hassell showed in
\cite{GRH2015} that an appropriatly rescaled spectral measure for
$S_h$ converges to the pushforward via the map $\mathbb{R} \lra
\mathbb{R} / 2 \pi \mathbb{Z} = \mathbb{S}^1$ of a homogeneous measure
$$
\nu = \left\{
  \begin{array}{cc}
c_1 \theta^{-(\gamma + 1)} & \mbox{ for } \theta > 0 \\
c_2 |\theta|^{-(\gamma + 1)} & \mbox{ for } \theta < 0
  \end{array}
\right.,
$$
on $\mathbb{R}$, where $\gamma = (d-1)/(\alpha - 1)$ and the $c_1,
c_2$ are determied by $v_0$.  It would be interesting to know if there
are circumstances under which equidistribution fails in the setting of
obstacle scattering.

\subsubsection*{Organisation of the paper}
In Section \ref{sec:dynamics}, we will recall a few facts about the
classical scattering dynamics, and its links with the interior
billiard dynamics. In Section \ref{sec: tools}, we will recall the
main tools we use to prove Theorem \ref{th: main theorem}. In Section \ref{sec:ProofProposition}, we prove Theorem \ref{th: main theorem} in the special case when $f$ is a polynomial vanishing at one. Finally, we prove
Theorems \ref{th: main theorem} and \ref{thm:scattering-phase} in
Section \ref{sec:proofs}.  The appendix contain rather elementary facts of
semiclassical analysis, and a proof of a resolution of identity formula on
the sphere.

\subsubsection*{Acknowledgements}  J.G.R. acknowledges the support of
the Australian Research Council through Discovery Grant DP180100589.
M.I. was funded by the LabEx IRMIA, and partially supported by the
Agence Nationale de la Recherche project GeRaSic (ANR-13-BS01-0007-01).
Both authors wish to thank the Australian Mathematical Sciences
Institute and the Mathematical Sciences Institute at the Australian
National University for their partial funding of the workshop
``Microlocal Analysis and its Applications in Spectral Theory,
Dynamical Systems, Inverse Problems, and PDE'' at which part this
project was completed.

\section{Classical scattering dynamics and interior dynamics}\label{sec:dynamics}
Let $\omega\in \Sph^{d-1}$ and $\eta\in \omega^\perp\subset \R^d$. We will always identify $(\omega,\eta)$ with a point in $T^*\Sph^{d-1}$.
Consider the line $L_{(\omega,\eta)}:=\{t\omega +\eta, t\in \R\}$. By strict convexity of $\partial \Omega$, it intersects $\partial \Omega$ in zero, one or two points. 
We define the interaction region,
\begin{equation}\label{eq: def I}
\mathcal{I}:= \{ (\omega,\eta)\in T^*\Sph^{d-1} ; L_{(\omega,\eta)}\cap \partial \Omega \text{ contains two points} \}.
\end{equation}
If $(\omega,\eta)\in \mathcal{I}$, then there exists $t_-<t_+$ such that $t_\pm\omega+\eta \in \partial \Omega$. We then set (see Figure \ref{fig: kappa})
\begin{figure}
    \center
\includegraphics[scale=0.4]{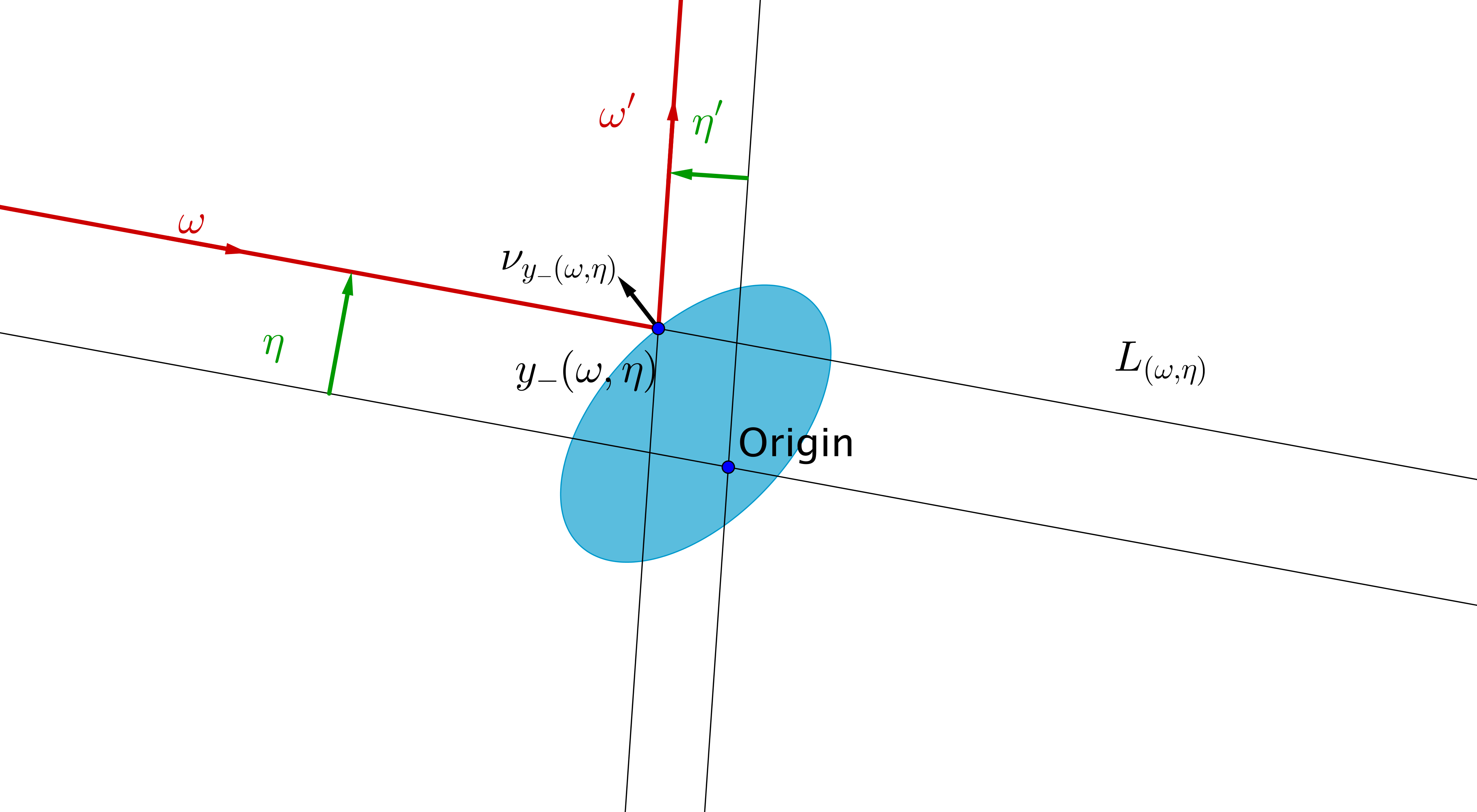}
    \caption{The construction of the scattering map $\kappa$.}\label{fig: kappa}
\end{figure}
\begin{equation*}
\begin{aligned}
y_\pm(\omega,\eta)&:= t_\pm\omega+\eta\in \partial \Omega\\
\omega'(\omega,\eta)&:= \omega- 2 (\omega\cdot \nu_{y_-(\omega,\eta)}) \nu_{y_-(\omega,\eta)}\\
\eta'(\omega,\eta) &:= y_-(\omega,\eta) - (\omega'\cdot x(\omega,\eta)) \omega',
\end{aligned}
\end{equation*}
where $\nu_y$ is the outward pointing normal vector at the point $y \in
\p \Omega$.
We then set
\begin{equation}
\kappa(\omega,\eta) = (\omega',\eta').\label{eq:the-scattering-map}
\end{equation}
If $(\omega,\eta)\notin \mathcal{I}$, we shall set $\kappa(\omega,\eta) = (\omega,\eta)$.
The map $\kappa$ may then be seen as a $C^0$ map $\kappa :
T^*\Sph^{d-1} \rightarrow T^*\Sph^{d-1}$, which is smooth (and even
symplectic) away from the glancing set $\partial \Omega$. 

Using Cauchy's surface area formula, it is straightforward that
\begin{equation}\label{eq:CauchyFormula}
\Vol(\mathcal{I}) = \Vol(\p \Omega) \omega_{d-1}.
\end{equation}

For $p\in
\mathbb{Z}\backslash\{0\}$, we will denote by $\mathcal{P}_p\subset
T^*\Sph^{d-1}$ the set of fixed points of $\kappa^p$. 
Note that we then have
$$\mathcal{I}= T^*\Sph^{d-1}\backslash \mathcal{P}_1,$$
and that
$\partial \mathcal{P}_1= \partial \mathcal{I}$ is exactly the `glancing set', i.e.\ the set of $(
\omega, \eta)$ such that $L_{(\omega,\eta)}\cap \partial \Omega$
consists of a single point.  We define
\begin{equation}
  \label{eq:non-trivial-periodic}
  \mathcal{P}_p' := \mathcal{P}_p  \setminus \mathcal{P}_1,
\end{equation}
the set of non-trivial glancing periodic points with period $p$, also an invariant subset.

The sets $\mathcal{P}'_p$ will play a central role in our proof, and
can be better understood in terms of the periodic points of the
interior billiard map, as follows. 
Consider the set $\mathcal{O}:=\{(y,\xi)\in S^*\partial \Omega;~ \xi\cdot \nu_y<0\}$. If $(y,\xi)\in \mathcal{O}$, there will be a unique $t>0$ such that $y+t\xi\in \partial \Omega$. We shall then write $y'(y,\xi)= y+ t\xi$, and $\xi'(y,\xi) = \xi - 2 (\xi\cdot \nu_{y'}) \nu_{y'}$.
 We have $(x',\xi')\in \mathcal{O}$, and we may define $\kappa_{int} : \mathcal{O}\rightarrow \mathcal{O}$ by
$\kappa_{int}(x,\xi) = (x',\xi')$. The map $\kappa_{int}$, and we
shall denote by $\mathcal{P}_p^{int}$ the set of periodic points of
period $p$ of $\kappa_{int}$. 

The following elementary lemma makes explicit the link between $\kappa$ and $\kappa_{int}$, as can be seen on Figure \ref{fig: billiard}.
\begin{lemma}
Let $(\omega,\eta)\in T^*\Sph^{d-1}\backslash \mathcal{P}_1$. We then have $$\kappa_{int} \big{(}(y_-(\omega,\eta), - \omega\big{)} = \big{(} y_-(\kappa(\omega,\eta), \omega'(\kappa(\omega,\eta)\big{)}.$$
\end{lemma}
\begin{figure}
    \center
\includegraphics[scale=0.4]{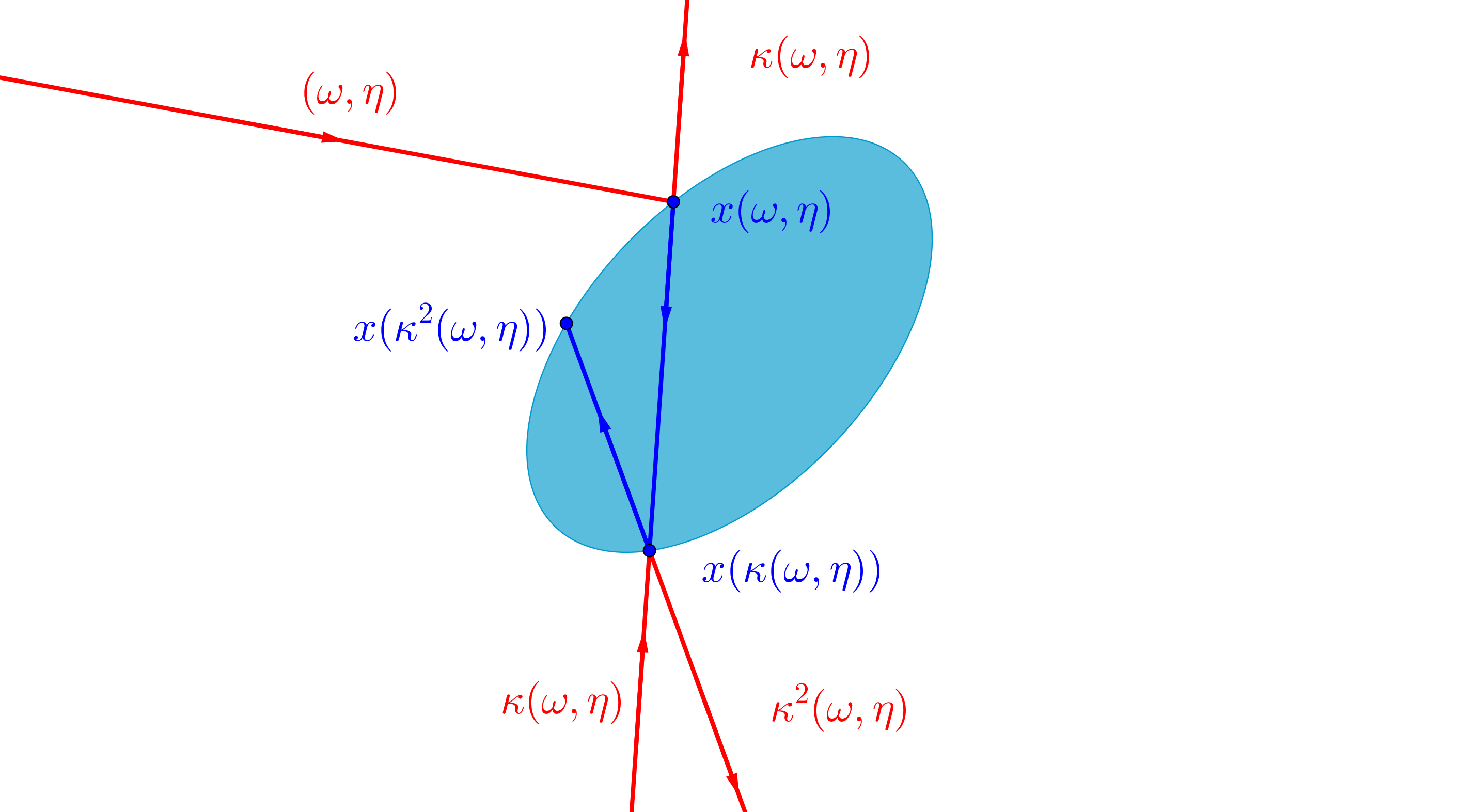}
    \caption{The scattering map and the billiard map.}\label{fig: billiard}
\end{figure}
As a consequence of this lemma, we see that $\mathcal{P}'_p$ is homeomorphic to $\mathcal{P}_p^{int}$.

\subsection*{The volume of the set of fixed points}

Let us denote by $\Vol$ the (symplectic) volume on $T^*\Sph^{d-1}$. We will always assume that we have
\begin{equation}\label{eq: hypVol}
\forall p\in \mathbb{Z}\backslash \{0\}, ~~ \Vol (\mathcal{P}'_p) = 0.
\end{equation}

This condition may of course be rephrased in terms of the dynamics of $\kappa_{int}$ on $\mathcal{O}$. If $\mu_\mathcal{O}$ is any Riemannian volume and $d_\mathcal{O}$ is any Riemannian distance on the manifold $\mathcal{O}$,
Equation (\ref{eq: hypVol}) is equivalent to
\begin{equation}\label{eq: hypVol2}
\forall p\in \mathbb{Z}\backslash \{0\},~~  \mu (\mathcal{P}^{int}_p) = 0,
\end{equation}

Condition (\ref{eq: hypVol2}) is conjectured to hold for all domains $\Omega\subset \R^d$, not necessarily convex. This conjecture, known as Ivrii's conjecture, has implications in terms of remainders for the Weyl's law for the eigenvalues of the Laplacian (see \cite{ivrii1980second}).
In the generic case, it was shown in 
\cite{petkov1988number} that $\mathcal{P}^{int}_p$ is finite for all $p\in \mathbb{Z}\backslash\{0\}$, so that (\ref{eq: hypVol2}) holds.
If the manifold $\partial \mathcal{O}$ is analytic, then the map $\kappa_{int}$ will be analytic, and we can show that (\ref{eq: hypVol2}) will hold (see for instance \cite{safarov1997asymptotic}).

\section{Tools for the proof of Proposition \ref{thm:trace-lemma}}\label{sec: tools}
Before proving Proposition \ref{thm:trace-lemma}, let us recall a few facts we will need in the proof.
\subsection{An integral representation for the scattering amplitude} 
The operator \begin{equation}
  \label{eq:scattering-matrix-and-amplitude}
  A(k) := S(k) - \Id.
\end{equation} can also be defined as follows.
Let $v(\cdot; \xi,k)$ be the unique solutions to
\begin{equation}
  \label{eq:dirichletexterior}
  \begin{split}
    (\Delta + k^{2}) v &= 0  \\
    v \rvert_{\p \Omega} &= -e^{i x \cdot \xi},
  \end{split}
\end{equation}
satisfying the Sommerfeld radiation condition.

$v$ may then be written as
$$v(|x|\omega;\xi,k) = |x|^{-(d-1)/2}e^{i k|x|} a(\omega,\xi,k) + O(|x|^{-(d+1)/2}).$$

One can show (see for instance \cite{HR1976}, page 381) that $A(k)$ is given by an integral kernel 
\begin{equation}
  \label{eq:big-a}
A(k)f(\omega)= \int_{\mathbb{S}^{d-1}} a(\omega, \theta, k) f(\theta) dVol_{\mathbb{S}^{d-1}}(\theta),
\end{equation}
where $a$ satisfies
\begin{align}
  \label{eq:little-a}
  a(\omega, \theta, k) &= \frac{1}{2i} k^{d-2} (2\pi)^{1-d} \int_{\p \Omega} e^{ik \omega
    \cdot y} \frac{\p}{\p \nu} e^{-ik\theta \cdot y} dVol_{\p
    \Omega}(y) \\
 &\qquad + \frac{1}{2i} k^{d-2} (2\pi)^{1-d} \int_{\p \Omega} e^{-ik \theta
    \cdot y} \frac{\p}{\p \nu} v(y, -k\omega) dVol_{\p
    \Omega}(y).
\end{align}

\subsection{The Kirchhoff approximation}
The function $\p_\nu v$ was studied in \cite{Melrose-Taylor-near-peak}, where the authors write 
\begin{equation}
  \label{eq:K-definition}
  \frac{\p v(x, -k\omega)}{\p \nu}  = K(\omega, x, k) e^{i k x \cdot \omega}.
\end{equation}
The main result we need from them can be summed up as follows. (The definition of the symbol classes $S_\delta$ is recalled in Appendix \ref{sec:Appendix}.)

\begin{theorem}[Melrose-Taylor, \cite{Melrose-Taylor-near-peak}]\label{th: meltay}
$$
K(\omega,x,k) = - i k | \nu_x \cdot \omega| + k E(\omega,x,k),
$$
where $\nu_x$ is the outward pointing normal vector at the point $x \in
\p \Omega$, and where $E$ satisfies
\begin{equation}\label{eq:symbolE}
E\in k^{-1/3} S_{1/3}(\Sph^{d-1}\times \partial \Omega).
\end{equation}
\end{theorem}

In particular, we have
\begin{equation}\label{eq:boundE}
||E(\cdot,\cdot,k) ||_{C^0}\leq C k^{-1/3}.
\end{equation}
We therefore have
\begin{equation}\label{eq : Kirchhoff amplitude}
\begin{aligned}
a(\omega, \theta, k) &= -\frac{1}{2} \Big{(}\frac{k}{2\pi}\Big{)}^{d-1} \int_{\p \Omega}\Big{(} e^{i k
  (\omega - \theta) \cdot y} ( - \nu_y \cdot \theta + |\nu_y \cdot
\omega| + E(\omega,y,k))\Big{)} dVol_{\p \Omega}(y).
\end{aligned}
\end{equation}

\subsection{The use of Gaussian states}
From now on, we fix a smooth compactly supported function 
$$\chi : [0,\infty)\rightarrow [0,1]$$ taking value $1$ in a neighborhood of $0$.
Let $(\omega_0,\eta_0) \in T^*\Sph^{d-1}$. We set
$$\phi_{\omega_0,\eta_0}(\omega;k):= \Big{(}\frac{k}{2\pi}\Big{)}^{(d-1)/4} \chi \big{(} k^{1/4}|\omega-\omega_0|\big{)} e^{- ik  \eta_0 \cdot(\omega-\omega_0)}   e^{-\frac{k}{2} |\omega-\omega_0|^2}.$$
Note that $\|\phi_{\omega_0,\eta_0}\|_{L^2} = 1+O(h^\infty)$. The term  $\chi \big{(} k^{1/4}|\omega-\omega_0|\big{)}$ is not very important here, and we could replace it by $ \chi \big{(} k^{p}|\omega-\omega_0|\big{)}$ for any $p\in (0, 1/2)$. It is here only to ensure that the integrals in (\ref{resolution}) and  (\ref{eq:trace-gaussian-states}) below makes sense.

The following lemma, whose proof can be found in \cite[\S 4]{ingremeau2016semi}, says that any function can be decomposed along the $\phi_{\omega_0,\eta_0}$ through a resolution of identity formula.

\begin{lemma}\label{lemma: resolution3}
1) Let $f\in L^2(\Sph^{d-1})$. We have
\begin{equation}\label{resolution}
f(\omega) = c_k \int_{T^* \Sph^{d-1}}\mathrm{d}\omega_0 \mathrm{d}\eta_0  \phi_{\omega_0,\eta_0}(\omega) \int_{\Sph^{d-1}} \mathrm{d}\omega'  \overline{ \phi_{\omega_0,\eta_0}(\omega')} f(\omega'),
\end{equation}
where $c_k$ is a parameter depending on $k$, with 
\begin{equation*}c_k= 2^{-(d-1)/2}\big{(}\frac{k}{2\pi} \big{)}^{d-1} + O_{k\rightarrow \infty}(k^{d-3/2}).
\end{equation*}

2) Let  $T \in \mathcal{L}(L^2(\mathbb{S}^{d-1}))$ a trace
class operator, then
\begin{equation}
\Tr (T) = c_k \int_{T^*\Sph^{d-1}}d\omega_0 d\eta_0 \langle \phi_{\omega_0,\eta_0}, T \phi_{\omega_0,\eta_0}\rangle_{L^2(\Sph^{d-1})}.
\label{eq:trace-gaussian-states}
\end{equation}
\end{lemma}

Let $\mathrm{d}$ be some Riemannian distance on $T^*\Sph^{d-1}$.
For any $\varepsilon>0$, the set 
\begin{equation}\label{eq: def C alpha}
G_\varepsilon=\{(\omega,\eta) \in
T^*\Sph^{d-1} ; \mathrm{d}\big{(}(\omega,\eta), \partial \mathcal{P}_1\big{)} < \varepsilon\}
\end{equation}
 has volume
$O(\varepsilon)$. Therefore, since $\| S(k) \|_{L^2 \lra L^2} = 1$, we have
\begin{equation}
  \label{eq:alph-bound}
I_\varepsilon (k) := c_k \int_{G_\varepsilon} d\omega_0 d\eta_0 \langle
\phi_{\omega_0,\eta_0}, (S_k^p-Id)
\phi_{\omega_0,\eta_0}\rangle_{L^2(\Sph^{d-1})} = O (\varepsilon k^{d-1}).
\end{equation}

\subsection{The action of the scattering matrix on a Gaussian state}
The main tool we use in the proof of Theorem \ref{th: main theorem} is the following proposition, which describes the action of the scattering matrix on a Gaussian state
\begin{proposition}\label{prop:PropagGaussian}
Let $(\omega_0,\eta_0)\in T^*\Sph^{d-1}\backslash G_{\varepsilon_0}$.
Let us write $(\omega_1,\eta_1):= \kappa (\omega_0,\eta_0)$. We have

\begin{equation}\label{eq:PropagGaussDecompo}
\begin{aligned}
&S_k \phi_{\omega_0,\eta_0}(\theta) = A_1(\theta; \omega_0,\eta_0,k) + A_2(\theta; \omega_0,\eta_0,k) \\
&\|A_2(\cdot; \omega_0,\eta_0,k)\|_{L^2}=O(k^{-1/3})\\
&A_1(\theta; \omega_0,\eta_0,k)= a_1(\theta; \omega_0;\eta_0) e^{ik\Phi_1(\theta; \omega_0,\eta_0)},
\end{aligned}
\end{equation}
where 
\begin{equation}\label{eq:SymbolAfterPropag}
a_1\in k^{(d-1)/4} S_{0},
\end{equation}
$\Phi_1 (\cdot ; \omega_0,\eta_0) \in C^\infty(\Sph^{d-1}; \C)$
 has a non-negative imaginary part vanishing only at $\omega_1$, and
 $\partial\Phi_1(\omega_1; \omega_0,\eta_0)=\eta_1$, and $\partial^2
 \Phi_1(\omega_1; \omega_0,\eta_0)$ has a positive definite imaginary
 part.
\end{proposition}

\begin{proof}
Thanks to (\ref{eq : Kirchhoff amplitude}, we have

\begin{equation}\label{eq:AlmostDoneWritingThisPaper}
\begin{aligned}
\big{(}S_k-Id\big{)} \phi_{\omega_0,\eta_0}(\theta) = &-\frac{1}{2} \Big{(}\frac{k}{2\pi}\Big{)}^{5(d-1)/4} \int_{\Sph^{d-1}}  \int_{\p \Omega}  \chi \big{(} k^{1/4}|\omega-\omega_0|\big{)} e^{-ik  \eta_0 \cdot \omega}   e^{-\frac{k}{2} |\omega-\omega_0|^2}\\
&\times \Big{(} e^{i k
  (\omega - \theta) \cdot y} ( - \nu_y \cdot \theta + |\nu_y \cdot
\omega| + E(\omega,y,k))\Big{)} dVol_{\p \Omega}(y) \mathrm{d}\omega.
\end{aligned}
\end{equation}

This is an oscillatory integral with a phase
\begin{equation}\label{eq:DefPhi}
\Phi(\omega,y; \theta) = -\eta_0 \cdot \omega + \frac{i}{2} |\omega-\omega_0|^2 +  (\omega - \theta) \cdot y.
\end{equation}

Let $y_\pm\in \partial \Omega$ be the points such that $y_\pm \in \eta_0 + \R \omega_0$ and $\pm \nu_{y_\pm} \cdot \omega_0 >0$. Let $\theta_\pm\in \Sph^{d-1}$ be such that $\omega_0-\theta_\pm\in \R \nu_{y_\pm}$ and $\theta_\pm \neq \omega_0$. We also set $\theta_0:= \omega_0$. Note that
\begin{equation}\label{eq:ItissofuntoworkinJessesoffice}
(\omega_1,\eta_1) = (\theta_-, \pi_{\theta_-^\perp} (y_-)).
\end{equation}

The phase $\Phi$ has a vanishing gradient and imaginary part at four points: $(\omega_0, y_+, \theta_0)$, $(\omega_0, y_-,\theta_0)$, $(\omega_0, y_+, \theta_+)$ and $(\omega_0,y_-,\theta_-)$.  

We will show below that the second derivative of $\Phi$ is non-degenerate at these critical points. In particular, this implies thanks to Lemma \ref{lemm: nonstat} that $\big{(}(S_k-Id) \phi_{\omega_0,\eta_0}\big{)}(\theta)= O(k^{-\infty}),$ unless, for some $j\in \{0,+,-\}$, we have $|\theta-\theta_j|= O_\varepsilon(k^{-1/2 + \varepsilon})$.

We shall write 
\begin{equation} \label{eq:lingo}
\begin{split}
\alpha_\pm&:=\alpha_\pm(\omega_0,\eta_0):= \omega_0\cdot \nu_{y_\pm}\\
\delta_\pm (\theta)&:= (\omega_0-\theta)\cdot \nu_{y_\pm},
\end{split}
\end{equation}
so that $\delta_\pm (\theta_\pm) = 2\alpha_\pm$, and $\delta_\pm (\omega_0)= 0$.

Let us take a partition of unity on $\p \Omega$:
$$1 = \chi_+ + \chi_-,$$
with $\chi_\pm\in C^\infty(\p\Omega)$, satisfying $\chi_\pm(y) = 1$ in a neighborhood of $y_\pm$.
We shall consider the two integrals
\[
\begin{aligned}
I_\pm(\theta)  = &-\frac{1}{2} \Big{(}\frac{k}{2\pi}\Big{)}^{5(d-1)/4} \int_{\Sph^{d-1}}  \int_{\p \Omega} \chi_\pm(y) \chi \big{(} k^{1/4}|\omega-\omega_0|\big{)} e^{-ik  \eta_0 \cdot \omega}   e^{-\frac{k}{2} |\omega-\omega_0|^2}\\
&\times \Big{(} e^{i k
  (\omega - \theta) \cdot y} ( - \nu_y \cdot \theta + |\nu_y \cdot
\omega| + E(\omega,y,k))\Big{)} dVol_{\p \Omega}(y) \mathrm{d}\omega.
\end{aligned}
\]

To analyse these integrals, let us introduce more convenient local coordinates.

\subsection*{Local Coordinates}
We shall write the points $y\in \p \Omega$ close to $y_\pm$ as
$$y =  y(u)= y_\pm + u + f_\pm(u),$$
where $u\in T_{y_\pm}\p \Omega$, $f_\pm(u)\in \R \nu_{y_\pm}$, $f_\pm(u) = -\frac{u \cdot M u}{2} \nu_{y_0} + O_{|u|\rightarrow 0}(|u|^3)$. Here, $M$ is the second fundamental form of $\p \Omega$ at $y_{\pm}$, and is therefore a positive definite symmetric matrix.

Similarly, we shall write the points $\omega\in \Sph^{d-1}$ close to $\omega_0$ as
$$\omega = \omega(v)= \omega_0 + v + g(v),$$
where $v\in \omega_0^\perp$, $g(v)\in \R\omega_0$, $g(v) = \frac{|v|^2}{2} \omega_0 + O_{|v|\rightarrow 0}(|v|^3)$.

Finally, for $j=0, +,-$, we shall write the points $\theta\in \Sph^{d-1}$ close to $\theta_j$, as $$\theta = \theta(w)= \theta_j + w + g(w),$$
where $w\in \theta_j^\perp$, $g(w)\in \R\theta_j$, $g(w) = \frac{|w|^2}{2} \theta_j + O_{|w|\rightarrow 0}(|w|^3)$.

In these coordinates, using the fact that $(\theta_j-\omega_0)\cdot u = 0$, we have
\[
\begin{aligned}
\Phi(u,v; \theta) &=  -\eta_0 \cdot (v +g(v)) +  \frac{i}{2}|v+g(v)|^2 + (\omega_0-\theta  + v + g(v)) \cdot (y_\pm + u + f(u)) \\
&= (\omega_0-\theta) \cdot y_{\pm} +(\theta_j-\theta)\cdot u + \delta_\pm (\theta) \frac{u\cdot M u}{2} -\frac{|v|^2}{2} (y_\pm \cdot \omega_0)  + \frac{i}{2} |v|^2 + u \cdot v + O(\max (|u|, |v|)^3).
\end{aligned}
\]

Since we are working away from the glancing set, the space $\omega_0^\perp\cap T_{y_\pm}\p \Omega$ is of dimension $d-2$. By definition of $\theta_j$, we have
\begin{equation*}
\omega_0^\perp\cap T_{y_\pm}\p \Omega = \theta_j^\perp\cap T_{y_\pm}\p \Omega
\end{equation*}
for $j\in \{0,+,-\}$. Starting from an orthonormal basis of $\omega_0^\perp\cap T_{y_\pm}\p \Omega$, we may find an orthonormal basis of $T_{y_\pm}\p \Omega$, an orthonormal basis of $\omega_0^\perp$ and an orthonormal basis of $\theta_j^\perp$ such that, if $u$, $v$ and $w$ are written in the coordinates associated to these bases, we have 
\begin{equation*}
\begin{aligned}
u\cdot v &= (\omega_0 \cdot \nu_{y_\pm})  u_1 v_1 + \sum_{j=2}^{d-1} u_j v_j = \langle u, D_\pm v \rangle_{d-1}\\
u\cdot w &=  \langle u, D_\pm w \rangle_{d-1},
\end{aligned}
\end{equation*}
where, notation as in \eqref{eq:lingo},
\begin{equation}\label{eq:deftildeI}
D_\pm= \begin{pmatrix}
\alpha_\pm & 0 & ... & 0 \\
0 & 1 & ...  &0  \\
& & ... \\
0 & 0 & ... & 1
\end{pmatrix},
\end{equation}
 and where we denote by $\langle \cdot , \cdot \rangle_{n}$ the canonical scalar product on $\R^{n}$.
 
We therefore have
\[
\begin{aligned}
\Phi(u,v; w)
&= (\omega_0-\theta(w)) \cdot y_{\pm} +\langle u, \big{(} D_\pm w
+O(|w|^2) \big{)}\rangle_{d-1} + \frac{1}{2}\left< (u,v),
\mathcal{M}(\theta(w)) \begin{pmatrix}
u\\
v
\end{pmatrix} \right>_{2d-2} \\
&\qquad + O(\max (|u|, |v|)^3),
\end{aligned}
\]

where 
\begin{equation*}
\mathcal{M}(\theta)= \begin{pmatrix}
\delta_\pm (\theta) M & D_\pm\\
D_\pm & (i-y_\pm \cdot \omega_0) Id
\end{pmatrix}.
\end{equation*}
 
We may then write

\[
\begin{aligned}
I_\pm(\theta)  = &-\frac{1}{2} \Big{(}\frac{k}{2\pi}\Big{)}^{5(d-1)/4} \int_{\R^{2(d-1)}} \chi_\pm(y(u)) \chi \big{(} k^{1/4}|v|\big{)} e^{ik \Phi(u,v,\theta)}\\
&\times \Big{(}( - \nu_{y(u)} \cdot \theta + |\nu_{y(u)} \cdot
\omega(v)| + E(\omega(v),y(u),k)) +R(u,v)\Big{)} \mathrm{d}u \mathrm{d}v +O(k^{-\infty}),
\end{aligned}
\]  
where  $R\in S_0$ comes from the Jacobian of the change of coordinates $(u,v)\mapsto (y,\omega)$, and satisfies $R(0,0)=0$.

\subsubsection{Behavior for $\theta$ close to $\omega_0$}

The matrix $\mathcal{M}(\omega_0)$ is invertible, with inverse 
\begin{equation*}
\mathcal{M}(\omega_0)^{-1} = 
\begin{pmatrix}
- ( i -y_\pm \cdot \omega_0) D_\pm^{-2}  & D_\pm^{-1} \\
D_\pm^{-1}  & \boldsymbol{0}_{d-1}
\end{pmatrix},
\end{equation*}
and we have, notation as in \eqref{eq:lingo},
\begin{equation*}\det(\mathcal{M}(\omega_0))= (-1)^{d-1} \alpha_\pm^{2}.
\end{equation*}

For $\theta$ close to $\omega_0$, $\mathcal{M}(\theta)$ is still invertible, with inverse
$$\mathcal{M}(\theta)^{-1} = \mathcal{M}(\omega_0)^{-1} \Big{(}\boldsymbol{1}_{d-1} + O\Big{(} \theta-\omega_0\Big{)}\Big{)}.$$

By applying Lemma \ref{lem:StrangePhiStat}, we obtain that 

\begin{equation*}
\begin{aligned}
I_\pm(\theta) &= \pm\frac{1}{2}
\Big{(}\frac{k}{2\pi}\Big{)}^{5(d-1)/4} e^{ik (\omega_0-\theta)\cdot
  y_{\pm}} e^{\frac{ik}{2}\big{(}(i-y_\pm \cdot \omega_0)
  |\theta-\theta_0|^2 + O(|\theta-\theta_0|^3)\big{)} } \\
&\qquad \times 
\Big{(}\frac{2\pi}{k}\Big{)}^{d-1} \frac{1}{|\alpha_\pm|} \big{(}
-\alpha_\pm + |\alpha_\pm| +O_{k\rightarrow \infty}(k^{-1/3}) \big{)}.
\end{aligned}
\end{equation*}

Recall that we assume in this paragraph that $|\theta-\theta_0| =O_{\varepsilon} \big{(}k^{-1/2+\varepsilon}\big{)}$ for all $\varepsilon>0$.
Noting that $\theta\cdot y_\pm = \theta\cdot \eta_0 + (\omega_0\cdot y_\pm) (\theta\cdot \omega_0)$,  we deduce that

\begin{equation*}
I_\pm(\theta)= \pm \frac{1}{2} \Big{(}\frac{k}{2\pi}\Big{)}^{(d-1)/4}  e^{-ik \theta\cdot \eta_0} e^{- \frac{k}{2}|\theta-\theta_0|^2} \frac{1}{|\alpha_\pm|}  \big{(}-\alpha_\pm + |\alpha_\pm| +O_{k\rightarrow \infty}(k^{-1/3}) \big{)}.
\end{equation*}

Therefore, we have for all $\theta$ close to $\omega_0$.
\begin{equation*}
\big{(}(S_k-Id) \phi_{\omega_0,\eta_0}\big{)}(\theta) = -\phi_{\omega_0,\eta_0}(\theta)\big{(} 1+ O_{L^2}(k^{-1/3})\big{)}.
\end{equation*}

\subsubsection{Behavior for $\theta$ close to $\theta_\pm$}
Note that the matrix 
$$N(\theta):= (i-y_\pm\cdot \omega_0)^{-1}Id - \delta_\pm(\theta) D_\pm^{-1} M D_\pm^{-1} $$
 is invertible, since the second term is real symmetric, so that it does not have $(i-y_\pm\cdot \omega_0)^{-1}$ in its spectrum. It is then straightforward to check that $\mathcal{M}(\theta)$ is invertible, with inverse
\begin{equation*}
\mathcal{M}(\theta)^{-1} = \begin{pmatrix}
- D_\pm^{-1} N(\theta)^{-1}D_\pm^{-1} &  (i-y_\pm\cdot \omega_0)^{-1} D_\pm^{-1} N(\theta)^{-1} \\
 (i-y_\pm\cdot \omega_0)^{-1} N(\theta)^{-1} D_\pm^{-1} &  (i-y_\pm\cdot \omega_0)^{-1} Id - (i-y_\pm\cdot \omega_0)^{-2}N^{-1}(\theta)  
\end{pmatrix}.
\end{equation*}

Applying Lemma \ref{lem:StrangePhiStat}, we obtain that for $|\theta-\theta_\pm|\leq k^{-1/2+\varepsilon}$, we have

\begin{equation*}
\begin{aligned}
I_\pm(\theta)= -\frac{1}{2} \Big{(}\frac{k}{2\pi}\Big{)}^{5(d-1)/4} \Big{(}\frac{2i \pi}{k}\Big{)}^{d-1} \det(\mathcal{M}(\theta))^{-1/2}   e^{ik \Phi_1^\pm(\theta; \omega_0,\eta_0)}\big{(} -\alpha_\pm + |\alpha_\pm| +O_{k\rightarrow \infty}(k^{-1/3}) \big{)}
\end{aligned}
\end{equation*}

where 
\begin{equation*}
\begin{aligned}
\Phi_1^\pm(\theta; \omega_0,\eta_0)
&= (\omega_0-\theta_\pm) \cdot y_{\pm} + (\theta_\pm-\theta) \cdot y_{\pm}  -\frac{1}{2} \pi_{\theta_\pm^\perp} (\theta-\theta_\pm) \cdot N(\theta)^{-1}\pi_{\theta_\pm^\perp} (\theta-\theta_\pm) + O(|\theta-\theta_\pm|^3).
\end{aligned}
\end{equation*}

Therefore, we have $$\|I_+\|_{L^2} = O(k^{-1/3}),~~I_-(\theta) = \big{(}S_k \phi_{\omega_0,\eta_0}\big{)}(\theta) + O_{L^2}(k^{-1/3}).$$

We therefore have
\begin{equation*}
\big{(}S_k \phi_{\omega_0,\eta_0}\big{)}(\theta) = a_1(\theta)  e^{ik \Phi_1^-(\theta; \omega_0,\eta_0)}+O_{L^2}(k^{-1/3}),
\end{equation*}
where
\begin{equation*}
a_1(\theta) := \Big{(}\frac{k}{2\pi}\Big{)}^{(d-1)/4} \det(\mathcal{M}(\theta))^{-1/2}\in k^{(d-1)/4} S_{0},
\end{equation*}
and $\Phi_1= \Phi_1^-$ satisfies the announced properties, thanks to (\ref{eq:ItissofuntoworkinJessesoffice}).
\end{proof}

\section{Proof of Theorem \ref{th: main theorem}}\label{sec:ProofProposition}

The main ingredient in the proof of Theorem \ref{th: main
  theorem} is a trace formula for powers of $S_k$:
\begin{proposition}\label{thm:trace-lemma}
Suppose that (\ref{eq: hypVol2}) holds. Let $p \in \mathbb{Z}$.  We have
\begin{equation}\label{eq: TracePowers}
   \mbox{Tr} \big{[} S_k^p - Id \Big{]}  = (-1)^p
\Vol(\p \Omega) \omega_{d-1} \Big{(}\frac{k}{2\pi}\Big{)}^{d-1} + o(k^{d - 1} ).
\end{equation}
In particular, for any trigonometric polynomial $P$ vanishing at $1$
and for the measure $\mu_k$ in \eqref{eq:measure}, as $k \to \infty$,
\begin{equation*}
\langle \mu_k, P\rangle =
\frac{\Vol(\p \Omega) \omega_{d-1}}{2\pi} \int_{\Sph^1} P(\theta) \mathrm{d}\theta + o(1).
\end{equation*}
\end{proposition}

Theorem \ref{th: main theorem} can be deduced from Proposition \ref{thm:trace-lemma} in exactly the same way as in \cite[\S 5]{I2016}. We refer the reader to this paper for the argument.

Before proving Proposition \ref{thm:trace-lemma}, we shall prove the following corollary of Proposition \ref{prop:PropagGaussian}.

\begin{corollary}\label{cor:2}
Let $p \in \N$, let $(\omega_0,\eta_0)\in T^*\Sph^{d-1}\backslash G_{\varepsilon_0}$. Let us write $(\omega_p,\eta_p):= \kappa^p(\omega_0,\eta_0)$.
 We have
\begin{equation*}
S^p_k \phi_{\omega_0,\eta_0}(\theta) = A_{1,p}(\theta; \omega_0,\eta_0,k) + A_{2,p}(\theta; \omega_0,\eta_0,k),
\end{equation*}
where for all $\varepsilon>0$, $$\|A_{2,p}(\cdot; \omega_0,\eta_0,k)\|_{L^2}= O_\varepsilon(k^{-1/3+\varepsilon}),$$
and $A_{1,p}(\theta; \omega_0,\eta_0,k)$ is such that we have
\begin{equation*}
\langle \phi_{\omega,\eta}, A_{1,p}(\cdot; \omega_0,\eta_0,k)\rangle = O(k^{-\infty}).
\end{equation*}
 for all $(\omega,\eta) \in T^*\Sph^{d-1}$ satisfying $\mathrm{d}((\omega,\eta),(\omega_p,\eta_p)) > k^{-1/2+\varepsilon}$ for some $\varepsilon>0$.
\end{corollary}
\begin{proof}
We shall prove this result by induction. For $p=1$, the result is an immediate corollary of Proposition \ref{prop:PropagGaussian} and of the non-stationary phase Lemma \ref{lemm: nonstat}. Suppose that we have proven the result for some $p\in \N$, and let us prove it for $p+1$. 
By assumption, we have $S_k^{p+1} \phi_{\omega_0,\eta_0}(\theta) = S_k A_{1,p}(\theta; \omega_0,\eta_0,k) + S_k A_{2,p}(\theta; \omega_0,\eta_0,k).$

By Lemma \ref{lemma: resolution3}, we have
\begin{equation*}A_{1,p}(\theta; \omega_0,\eta_0,k) = c_k \int_{T^*\Sph^{d-1}}\mathrm{d}\omega \mathrm{d}\eta  \phi_{\omega,\eta}(\theta) \langle A_{1,p}(\cdot; \omega_0,\eta_0,k), \phi_{\omega,\eta} \rangle,
\end{equation*}
where $c_k$ is equivalent to some power of $k$.

We therefore have
\begin{equation*}
\begin{aligned}
S_k^{p+1} \phi_{\omega_0,\eta_0}&= c_k \int_{T^*\Sph^{d-1}}\mathrm{d}\omega  \mathrm{d}\eta  S_k\phi_{\omega,\eta} \langle A_{1,p}(\cdot; \omega_0,\eta_0,k), \phi_{\omega,\eta} \rangle + S_k A_{2,p}(\theta; \omega_0,\eta_0,k)\\
&=  c_k \int_{T^*\Sph^{d-1}}\mathrm{d}\omega  \mathrm{d}\eta  A_{1,1}(\cdot ; \omega,\eta,k)\langle A_{1,p}(\cdot; \omega_0,\eta_0,k), \phi_{\omega,\eta} \rangle\\
&+  c_k \int_{T^*\Sph^{d-1}}\mathrm{d}\omega \mathrm{d}\eta   A_{2,1}(\cdot ; \omega,\eta,k) \langle A_{1,p}(\cdot; \omega_0,\eta_0,k), \phi_{\omega,\eta} \rangle + S_k A_{2,p}(\theta; \omega_0,\eta_0,k).
\end{aligned}
\end{equation*}

We shall write
\begin{equation*}
\begin{aligned}
A_{1,p+1}(\theta; \omega_0,\eta_0,k)&:= c_k \int_{T^*\Sph^{d-1}}\mathrm{d}\omega \mathrm{d}\eta  A_{1,1}(\theta ; \omega,\eta,k)\langle A_{1,p}(\cdot; \omega_0,\eta_0,k), \phi_{\omega,\eta} \rangle\\
A_{2,p+1}(\theta; \omega_0,\eta_0,k)&:= c_k \int_{T^*\Sph^{d-1}}\mathrm{d}\omega  \mathrm{d}\eta   A_{2,1}(\theta ; \omega,\eta,k) \langle A_{1,p}(\cdot; \omega_0,\eta_0,k), \phi_{\omega,\eta} \rangle + S_k A_{2,p}(\theta; \omega_0,\eta_0,k).
\end{aligned}
\end{equation*}

In $A_{2,p+1}(\theta; \omega_0,\eta_0,k)$, the term $S_k A_{2,p}(\theta; \omega_0,\eta_0,k)$ is small in $L^2$ norm by recurrence hypothesis, since $S_k$ is unitary. For the other term, we note that the term $\langle A_{1,p}(\cdot; \omega_0,\eta_0,k), \phi_{\omega,\eta} \rangle$ in the integrand is $O(k^{-\infty})$ as soon as $(\omega,\eta)$ is at a distance larger than $k^{-1/2+\varepsilon}$ from $(\omega_p,\eta_p)$. Hence, the integrand is $O(k^{-\infty})$  away from a set of volume $O(k^{1-d+ C\varepsilon})$ for some $C>0$. On this set, the integrand has an $L^2_\theta$ norm bounded by $O(k^{-1/3+\varepsilon})$. Therefore, $\|A_{2,p+1}\|_{L^2} = O(k^{-1/3+C'\varepsilon})$.

As for $A_{1,p+1}$, we have
\begin{equation*}
\begin{aligned}
\langle \phi_{\omega,\eta}, A_{1,p+1}(\theta; \omega_0,\eta_0,k)\rangle
&=  c_k \int_{T^*\Sph^{d-1}}\mathrm{d}\omega' \mathrm{d}\eta'  \langle A_{1,1}(\cdot ; \omega',\eta',k), \phi_{\omega,\eta} \rangle \langle A_{1,p}(\cdot; \omega_0,\eta_0,k), \phi_{\omega',\eta'} \rangle.
\end{aligned}
\end{equation*}

Now, by the induction hypothesis, we know that the above integrand is $O(k^{-\infty})$ unless $\mathrm{d}(\kappa(\omega',\eta'), (\omega_p,\eta_p))=O\big{(} k^{-1/2+\varepsilon}\big{)}$ and $\mathrm{d}((\omega,\eta), (\omega',\eta'))= O\big{(}k^{-1/2+\varepsilon}\big{)}$. The result follows.
\end{proof}

We are now ready to prove Proposition \ref{thm:trace-lemma}.

\begin{proof}[Proof of Proposition \ref{thm:trace-lemma}]
First of all, let us note that it is enough to show the result for $p>0$. Indeed, since $S(k)$ is unitary, we have
\begin{equation*}
\begin{aligned}
\Tr \big{(} S(k)^{-p}-Id\big{)} &= \sum_n \langle e_n, ( S(k)^{-p}-Id) e_n \rangle~~\text{ for any orthonormal basis } (e_n)\\
&= \sum_n \langle (S(k)^p - Id) e_n, e_n \rangle\\
&= \Tr (S(k)^p-Id).
\end{aligned}
\end{equation*}

Therefore, let us fix from now on $p\geq 1$. By (\ref{eq:trace-gaussian-states}), we have 
\begin{equation}\label{eq:decompoTrace}
\begin{aligned}
  \mbox{Tr} \big{[} S_k^p - Id \big{]} &= c_k \int_{\mathcal{P}_1\backslash G_{\varepsilon_0}} \langle \phi_{\omega,\eta}, (S_k^p - Id) \phi_{\omega,\eta} \rangle \mathrm{d}\omega \mathrm{d}\eta+ c_k \int_{G_{\varepsilon_0}} \langle \phi_{\omega,\eta}, (S_k^p - Id) \phi_{\omega,\eta} \rangle \mathrm{d}\omega \mathrm{d}\eta \\
  &+  c_k \int_{\mathcal{I}\backslash G_{\varepsilon_0}} \langle \phi_{\omega,\eta}, (S_k^p - Id) \phi_{\omega,\eta} \rangle \mathrm{d}\omega \mathrm{d}\eta.
  \end{aligned}
  \end{equation}

By (\ref{eq:alph-bound}), the second term in the right-hand side of (\ref{eq:decompoTrace}) is $O(\varepsilon_0 k^{d-1}).$

To deal with the first term, we note that, when computing $\big{(}S_k-Id\big{)} \phi_{\omega_0,\eta_0}(\theta)$, the phase $\Phi$ satisfies $|\partial_{\omega} \Phi| = |\pi_{\omega} (\eta-y)| + |\omega-\omega_0| \geq C \mathrm{d}((\omega,\eta),\mathcal{I})$. Therefore, Lemma \ref{lemm: nonstat} implies that we have
$$\big{(}S_k-Id\big{)} \phi_{\omega_0,\eta_0}(\theta)= O\big{(} \big{(}k \mathrm{d}((\omega_0,\eta_0)\big{)}^{-\infty}\big{)}.$$
Therefore, the first term in the right-hand side of (\ref{eq:decompoTrace}) is $O(k^{-\infty}).$

We now compute
\begin{equation*}
\begin{aligned}
&c_k \int_{\mathcal{I}\backslash G_{\varepsilon_0}} \langle \phi_{\omega,\eta}, S_k^p \phi_{\omega,\eta} \rangle \mathrm{d}\omega \mathrm{d}\eta\\
&= c_k \int_{\mathcal{I}\backslash G_{\varepsilon_0}} \langle \phi_{\omega,\eta}, A_{1,p}(\cdot; \omega,\eta,k)+ A_{2,p}(\cdot; \omega,\eta,k)  \rangle \mathrm{d}\omega \mathrm{d}\eta \\
&=  c_k \int_{\mathcal{I}\backslash G_{\varepsilon_0}} \langle \phi_{\omega,\eta}, A_{1,p}(\cdot; \omega,\eta,k)  \rangle \mathrm{d}\omega \mathrm{d}\eta + O_\varepsilon(k^{d-1-1/3+\varepsilon}).
\end{aligned}
\end{equation*}
Now, by Corollary \ref{cor:2}, the integrand is $O(k^{-\infty})$ as soon as $\mathrm{d}(\kappa^p(\omega,\eta),(\omega,\eta)) > k^{-1/2+\varepsilon}$. But, by (\ref{eq: hypVol2}), the volume of the set of $(\omega,\eta)$ satisfying $\mathrm{d}(\kappa^p(\omega,\eta),(\omega,\eta)) > k^{-1/2+\varepsilon}$ is a $o_{k\rightarrow \infty}(1)$. All in all, we obtain that 
\begin{equation*}
\begin{aligned}
&c_k \int_{\mathcal{I}\backslash G_{\varepsilon_0}} \langle \phi_{\omega,\eta}, S_k^p \phi_{\omega,\eta} \rangle \mathrm{d}\omega \mathrm{d}\eta = o_{k\rightarrow \infty}(k^{d-1})
\end{aligned}
\end{equation*}

Therefore, (\ref{eq:decompoTrace}) gives us 
$$  \mbox{Tr} \big{[} S_k^p - Id \big{]} = c_k \int_{\mathcal{I}\backslash G_{\varepsilon_0}} \langle \phi_{\omega,\eta}, \phi_{\omega,\eta} \rangle \mathrm{d}\omega \mathrm{d}\eta + o_{k\rightarrow \infty}(k^{d-1}),$$
and the result follows since this is true for all $\varepsilon_0>0$.
\end{proof}

\section{Proof of theorem \ref{thm:scattering-phase}} \label{sec:proofs}
We now give our alternative proof of the scattering phase asymptotics
in Theorem \ref{thm:scattering-phase}.  We begin by recalling that the
scattering phase $s(k)$ can be defined continuously in such a way that
$\lim_{k \to 0^+} s(k) = 0$ and thus defined is in fact smooth for all $k >
0$.  We define the `reduced' scattering phase by the sum
$$
s_{2\pi}(k) = -\sum_{e^{i\beta_{k,n}} \in \sigma(S(k))} \beta_{k,n}
$$
where the logarithms of the eigenvalues, the $\beta_{k,n}$ are chosen
to take values in $(-2\pi, 0]$.  For fixed $k$ the eigenvalues
accumulate at $1$ from the bottom half plane and thus contribute
positive values to the sum, which is nonetheless finite.  A result of
Eckmann-Pillet \cite{EP1995} shows that eigenvalues approach $1$ with
positive imaginary part if and only if $k$ approaches a Dirichlet
eigenvalue of $\Omega$.  In fact, with $N_D(k)$ as in
\eqref{eq:counting-function}, we have
$$
s(k) = 2 \pi N_D(k) + s_{2\pi}(k).
$$
Under the hypothesis that the measure of the periodic billiard
trajectories in $\Omega$ is zero, it is known \cite{ivrii1980second}
that 
$$
N_D(k) =  \frac{\omega_d}{(2\pi)^{d}} \Vol(\Omega) k^d -
  \frac{\omega_{d - 1}}{4 (2\pi)^{d-1}}  \Vol(\p \Omega) k^{d -1} + o(k^{d-1}).
$$

 We claim that
$$
s_{2\pi}(k) = s(k) - N_D(k) = \frac{\omega_{d - 1}}{2 (2\pi)^{d-2}}  \Vol(\p \Omega) k^{d -1} + o(k^{d-1}).
$$
We will prove this by breaking up the unit circle into $M \in \mathbb{N}$ sectors of
size $2\pi/M$ estimating the sum defining $s_{2\pi}(k)$ in these
sectors.  Namely, let $A_{M,k}(j) := \{- 2\pi j/M <  \beta_{k,n} \leq - 2
  \pi (j + 1)/M \}$ and $\alpha_{M, k}(j) :=  -\sum_{A_{M, k}(j)} 
  \beta_{M,k}$, so that
$$
s_{2\pi}(k) = \sum_{j = 0}^{M - 1} \alpha_{M, k}(j).
$$
We begin with $j = 0$, which is distinct from $j > 0$ since there are
infinitely many phase shifts in $A_{M, k}(0)$.  We are going to show that
$$
|a_{M, k}(0)| \le \frac{C}{M} k^{d - 1}.
$$

Thanks to equation (2.3) in \cite{christiansen2015sharp}  (which relies on the methods developed in \cite{zworski1989sharp}), we have that there exists $C>0$ independent of $k$ and $n$ such that
\begin{equation}\label{eq:christ}
|e^{i\beta_{k,n}}-1|\leq Ck^d \exp \Big{(}  Ck- \frac{n^{1/(d-1)}}{C}\Big{)}.
\end{equation}

Let us write
\begin{equation*}
B_{M,k}:= \Big{[} \frac{2^{-(j+1)}}{M}; \frac{2^{-j }}{M}\Big{)}
\end{equation*}

Using \eqref{eq:christ} and a constant $C> 0$ whose value
changes from line to line, we see that 
\begin{equation}\label{eq:1}
  \begin{split}
    a_{M, k}(0) &\le \sum_{j = 0}^\infty \sum_{
      |(2\pi)^{-1}\beta_{k,n}|\in B_{M,k}} | e^{i \beta_{k,n}} - 1 | \\
&\le \sum_{j = 0}^\infty  \lp \frac{1}{M} 2^{-j}  \rp  \sum_{
     |(2\pi)^{-1}\beta_{k,n}| \in B_{M,k}} 1 \\
&\le \sum_{j = 0}^\infty  \lp \frac{1}{M} 2^{-j}  \rp C (k + (j +
1)/k)^{d-1} \\
&\le \frac{k^{d-1}}{M} \sum_{j = 0}^\infty  \lp 2^{-j}  \rp C (1 + (j
+ 1)/k^2)^{d-1} \le C\frac{ k^{d-1}}{M}.
  \end{split}
\end{equation}

For $j > 0$, we estimate $\alpha_{M, k}(j)$ from above and below, and clearly
\begin{equation}
  \begin{split}
 \frac{2\pi j}{M} | A_{M, k}(j) |    \le a_{M, k}(j) \le \frac{ 2 \pi (j +
 1)}{M} | A_{M, k}(j) |
  \end{split}
\end{equation}
It follows from the \eqref{eq:number}, since our sectors are size
$2\pi/M$, that for $0 < j \le M -1$, for  any $\delta > 0$ and
$k > k_{M, \delta}$, there is a constant $C > 0$ such that
$$
\lp \frac{\omega_{d-1}}{(2\pi)^{d-1}} \frac{2\pi}{M}
 \Vol(\p \Omega) - \delta \rp  k^{d-1} \le  |A_{M, k}(j)| \le \lp \frac{\omega_{d-1}}{(2\pi)^{d-1}} \frac{2\pi}{M}
 \Vol(\p \Omega) + \delta \rp  k^{d-1}
$$

Since $\sum_{j = 1}^{M-1} (j + 1) =  M(M - 1)/2$
we have
\begin{equation}
\begin{gathered}
  \lp \frac{\omega_{d-1}}{(2\pi)^{d-2}} \frac 12 \Vol(\p \Omega) -
  C(1/M + \delta M) \rp k^{d-1} \le s_{M, k}  \\ s_{M, k} \le \lp
  \frac{\omega_{d-1}}{(2\pi)^{d-2}} \frac 12 \Vol(\p \Omega) + C(1/M +
  \delta M) \rp k^{d-1},
\end{gathered}
\end{equation}
and thus
\begin{equation}
\begin{split}
  \limsup_{k \to \infty} s(k) k^{1-d} &\le
  \frac{\omega_{d-1}}{(2\pi)^{d-2}} \frac 12 \Vol(\p \Omega) + C(1/M +
  \delta M), \\ \liminf_{k \to \infty} s(k) k^{1-d} &\ge
  \frac{\omega_{d-1}}{(2\pi)^{d-2}} \frac 12 \Vol(\p \Omega) - C(1/M +
  \delta M)
\end{split}
\end{equation}
for any $M \in \mathbb{N}, \delta > 0$.  Taking $\delta = 1/M^2$ and sending
$M \to \infty$ gives the result.

\appendix

\section{Symbol classes and stationary phase}\label{sec:Appendix}
Let $X$ be a compact manifold, and let $a=(a_k)_{k>0}$ be a family of $\C$-valued functions in $C^{\infty}(X)$, and let $0\leq\delta<1/2$.
We shall write that $a\in S_{\delta}(X)$ if
\begin{equation}\label{eq:defsymbol}
\forall \beta \in \mathbb{N}^{d-1}, \exists C_{\beta} \text{ such that } \forall x\in X, |\partial^\beta  a(x,k)| \leq C_{\beta}k^{|\beta| \delta}.
\end{equation}

The following lemma follows from \cite[Lemma 7.7.1]{Hvol1}.

\begin{lemma}\label{lemm: nonstat}
Let $\delta<1/2$.
Let $a\in S_{\delta}(X)$, and $\Phi\in S_{0}(X)$. Suppose that there exists $C, \varepsilon>0$ such that for all $x$ in the support of $a_k$, we have $|\partial \Phi_k(x)|\geq C k^{-1/2+\varepsilon}$. We then have
$$\int_X e^{ik \Phi_k(x)} a_k(x) dx = O(k^{-\infty}).$$
\end{lemma}

The following stationary phase result is a variant over \cite[Theorem 7.7.5]{Hvol1}, but we recall its proof for completeness. Note that we do not assume that the derivative of the phase vanishes at the origin, but only that $\partial \varphi_k(0)= O(k^{-1/2+\varepsilon})$ for all $\varepsilon>0$.

\begin{lemma}\label{lem:StrangePhiStat}
Let  $\varphi\in S_0 (\R^n)$ have a positive imaginary part, and be such that
\begin{itemize}
\item There exists $C>0$ such that for all $\varepsilon>0$, $\|\partial \varphi_k(x)\| = O(k^{-1/2+\varepsilon}) \Leftrightarrow \| x\| = O(k^{-1/2+ C \varepsilon})$.

\item $\partial^2_x\varphi_k(0)$ is invertible for every $k$, and there exists  $c>0$ independent of $k$ such that $$\big{\|}\big{(}\partial^2_x\varphi_k(0)\big{)}^{-1}\big{\|}\leq c.$$
\end{itemize}

Let $a\in S_{1/3}(\R^n)$. Let us write $I(a):= \int_{\R^n} e^{ik \varphi_k(x)} a(x) \mathrm{d}x$. We have$$I(a) = e^{ik \varphi_k(0)} e^{- \frac{ik}{2} \partial \varphi_k(0)\cdot \partial^2 \varphi_k(0)^{-1} \partial \varphi_k(0)} \Big{(}\det \Big{(} \frac{ k \partial^2\varphi_k(0)}{2\pi i}\Big{)}\Big{)}^{-1/2}\big{(} a(\partial^2 \varphi_k(0)^{-1} \partial \varphi_k(0)) + O_{\varepsilon}(k^{-1/3}) \big{)},$$
where the square-root of the determinant is defined as in \cite[\S 3.4]{Hvol1}.
\end{lemma}

In particular, if $a\in S_0(\R^n)$, we have $a_k(\partial^2 \varphi_k(0)^{-1} \partial \varphi_k(0))= a_k(0) + O(k^{-1/3})$, so that we have
$$I(a)= e^{ik \varphi_k(0)} e^{- \frac{ik}{2} \partial \varphi_k(0)\cdot \partial^2 \varphi_k(0)^{-1} \partial \varphi_k(0)}  \Big{(}\det \Big{(} \frac{ k \partial^2\varphi_k(0)}{2\pi i}\Big{)}\Big{)}^{-1/2} \big{(} a_k(0) + O_{\varepsilon}(k^{-1/3})\big{)}.$$

\begin{proof}
Let $\chi\in C_c^\infty(\R^n)$ be such that $\chi(x)=1$ if $|x|<1/2$, and $\varepsilon>0$. Set $\chi_k(x) := \chi ( x k^{1/2-\varepsilon})$. We may write
$$\int_{\R^n} e^{ik \varphi_k(x)} a_k(x) \mathrm{d}x =\int_{\R^n} \chi_k(x) e^{ik \varphi_k(x)} a_k(x) \mathrm{d}x + \int_{\R^n} (1-\chi_k(x)) e^{ik \varphi_k(x)} a(x) \mathrm{d}x.$$

By Lemma \ref{lemm: nonstat} and the assumption we made on $\varphi$, the second integral is $O(k^{-\infty})$.

We have
\begin{equation*}
\chi_k(x) e^{ik \big{(} \varphi_k(x) -  \varphi_k (0) - x\cdot \partial \varphi_k(0) - \frac{1}{2} x\cdot \partial^2 \varphi_k (0) x \big{)}} \in S_{1/2-\varepsilon},
\end{equation*}
and 
\begin{equation}\label{eq:NiceDerivatives}
|x|=o(k^{-1/2+\varepsilon})\Longrightarrow \forall n\in \N ~~ \partial^n \Big{(} \chi_k(x) e^{ik \big{(} \varphi_k(x) -  \varphi_k (0) - x\cdot \partial\varphi_k(0) - \frac{1}{2} x\cdot \partial^2 \varphi_k (0) x \big{)}}\Big{)} = O(k^{(n-1)/2+n\varepsilon)}).
\end{equation}
Indeed, all the derivatives of $\chi_k$ will vanish for such $x$, and, when differentiating $e^{ik \big{(} \varphi_k(x) -  \varphi_k (0) - x\cdot \partial \varphi_k(0) - \frac{1}{2} x\cdot \partial^2 \varphi_k (0) x \big{)}}$ once, we get a factor of size $O(k^{\varepsilon} )$. Differentiating again, each derivation makes the function grow at most of a factor $k^{1/2+\varepsilon}$.

Let us write $b_k(x) := a_k(x) \chi_k(x) e^{ik \big{(} \varphi_k(x) -  \varphi_k (0) - x\cdot\partial \varphi_k(0) - \frac{1}{2} x\cdot \partial^2 \varphi_k (0) x \big{)}}$, so that $b_k \in S_{1/2-\varepsilon}$.
We have

\begin{equation*}
I(a)= e^{ik \varphi_k(0)}\int_{\R^n} e^{ik x\cdot \partial \varphi_k(0)} e^{\frac{ik}{2} x\cdot \partial^2\varphi_k(0) x} b_k(x) \mathrm{d}x. 
\end{equation*}

Let us write $A=A(k)= \partial^2\varphi(0)$, and $\xi_0= \xi_0(k) = \partial\varphi(0)$. We have 

$$\frac{1}{2} x\cdot Ax + x\cdot \xi_0 = \frac{1}{2} (x+A^{-1}\xi_0) \cdot A (x+A^{-1}\xi_0) - \frac{1}{2} \xi_0\cdot A^{-1} \xi_0,$$
so that, setting $y= x+A^{-1}\xi_0$, we have
\begin{equation*}
I(a)= e^{ik \varphi_k(0)} e^{- \frac{ik}{2} \xi_0\cdot A^{-1} \xi_0} \int_{\R^n} e^{\frac{ik}{2} y\cdot A y} b_k(y- A^{-1} \xi_0) \mathrm{d}y. 
\end{equation*}

Writing $\tilde{b}_k(y)= b_k(y-A^{-1}\xi_0)$ and $D=-i \partial$, we therefore have
\begin{equation*}
\begin{aligned}
I(a)&= e^{ik \varphi_k(0)} e^{- \frac{ik}{2} \xi_0\cdot A^{-1} \xi_0} \int_{\R^n} e^{\frac{ik}{2} y\cdot A y} \tilde{b}_k(y) \mathrm{d}y.
\end{aligned}
\end{equation*}

Now, we have
$$e^{\frac{ik}{2} x\cdot A x} =  \Big{(}\det \Big{(} \frac{2\pi k A}{i}\Big{)}\Big{)}^{-1/2}
\mathcal{F} \Big{(}e^{-\frac{i}{2k}\langle \cdot, A^{-1}\cdot\rangle}\Big{)},$$
so that by Plancherel's equality,
\begin{equation*}
I(a)=  e^{ik \varphi_k(0)} e^{- \frac{ik}{2} \xi_0\cdot A^{-1} \xi_0} \Big{(}\det \Big{(} \frac{2\pi k A}{i}\Big{)}\Big{)}^{-1/2} \int_{\R^n} e^{-\frac{i}{2k} \xi \cdot A^{-1}\xi} \mathcal{F}\big{(} \tilde{b}_k \big{)} (\xi) \mathrm{d}\xi.
\end{equation*}

Thanks to \cite[Theorem 7.6.5]{Hvol1}, this quantity is equal to
\begin{equation*}
\begin{aligned}
I(a)=  e^{ik \varphi_k(0)} e^{- \frac{ik}{2} \xi_0\cdot A^{-1} \xi_0} \Big{(}&\det \Big{(} \frac{i k A}{2\pi}\Big{)}\Big{)}^{-1/2} \Big{[} \sum_{j=0}^N \frac{1}{j!}  \Big{(}\Big{(}\frac{i}{2k} D\cdot A^{-1} D\Big{)}^j b_k\Big{)} (-A^{-1}\xi_0) \\
&+ O\Big{(} \Big{(} \frac{\|A\|^{-1}}{k}\Big{)}^N \sum_{|\gamma|\leq \frac{n}{2} +1 +2N} \|D^\gamma b_k\|_{C^0} \Big{)}\Big{]}.
\end{aligned}
\end{equation*}
Since $b_k\in S_{1/2-\varepsilon}$ and $\|A\|^{-1}= O(1)$, the remainder can be made smaller than any power of $k^{-1}$ by taking $N$ large enough.

Using (\ref{eq:NiceDerivatives}), we see that there exists $N\in \mathbb{N}$ such that
\begin{equation*}
I(a)=  e^{ik \varphi(0)} e^{- \frac{ik}{2} \xi_0\cdot A^{-1} \xi_0}  \Big{(}\det \Big{(} \frac{k A}{2\pi  i}\Big{)}\Big{)}^{-1/2} \Big{[} \sum_{j=0}^{N} \frac{1}{j!}  \Big{(}\Big{(}\frac{i}{2k} D\cdot A^{-1} D\Big{)}^j a\Big{)} (-A^{-1}\xi_0) + O\big{(} k^{-1/3}\big{)}\Big{]}.
\end{equation*}

The first term in the sum is $a(-A^{-1}\xi_0)$, and the following ones are $O(k^{-1/3})$. The statement follows.
\end{proof}

\bibliographystyle{alpha}

\end{document}